\documentclass{amsart}
  \usepackage{amsmath}
\usepackage{amssymb}
\usepackage{amsfonts}
 \usepackage{eucal}
      \makeatletter
     \def\section{\@startsection{section}{1}%
     \z@{.7\linespacing\@plus\linespacing}{.5\linespacing}%
     {\bfseries      \centering
     }}
     \def\@secnumfont{\bfseries}
     \makeatother
  \usepackage{a4wide}
\newtheorem{theorem}{Theorem}[section]
\newtheorem{lemma}[theorem]{Lemma}

\newtheorem{proposition}[theorem]{Proposition}
\theoremstyle{definition}

\newtheorem{problem}[theorem]{Problem}

\newtheorem{remark}[theorem]{Remark}
\theoremstyle{remarks}
\newtheorem{remarks}[theorem]{Remarks}

\numberwithin{equation}{section}

\def \a{{\alpha}}
\def \b{{\beta}}
\def \D{{\Delta}}

\def \d{{\delta}}
\def \e{{\varepsilon}}

\def \g{{\gamma}}

\def \k{{\kappa}}
\def \l{{\lambda}}

\def \o{{\omega}}

\def \p{{\varphi}}
\def \t{{\vartheta}}

\def \m{{\mu}}
\def \s{{\sigma}}

\def \N{{\bf N}}

\def \P{{\bf P}}

\def \qq{{\qquad}}
\def \R{{\bf R}}

 \def \Z{{\bf Z}}
 \def \dd{{\rm d}}

\def \noi{{\noindent}}

\def\E{{\mathbb E \,}}

\def\P{{\mathbb P}}

\def\R{{\mathbb R}}

\def\Z{{\mathbb Z}}

\def\N{{\mathbb N}}

\def\C{{\mathbb C}}

at  8 pt
\scrollmode

   at 10,4  pt at  8,2 pt
   \scrollmode
 at 10,2  pt

\font\sevenrm =cmr10 at  7  pt
      \title[Local Limit Theorems]
 {Local Limit Theorems in some Random models from  Number Theory}
  \author{
Rita   Giuliano  and    Michel   Weber}\address{IRMA, UMR 7501, Universit\'e
Louis-Pasteur,   7  rue Ren\'e Descartes, 67084
Strasbourg Cedex, France.
 E-mail:    {\tt  michel.weber@math.unistra.fr
 }}
  \address{Dipartimento di Matematica, Via F. Buonarroti  2, 56127 Pisa, Italy.   E-mail: {\tt  giuliano@dm.unipi.it} }

\begin{document}

%%%%%%%%

\begin{abstract}We study the local limit theorem for weighted sums of Bernoulli  variables. We show on examples that this  is an important question  in the general theory of the local limit theorem, and which turns up to be not  well explored. The examples we consider  arise from standard random models used in arithmetical number theory.    We next use the characteristic function method to prove new  local limit theorems for weighted sums of Bernoulli  variables.  Further, we give an application  of the almost sure local limit theorem to a representation problem in additive number theory due to Burr, using an appropriate random model. We also give a simple example showing that  the local limit theorem, in its standard form, fails to be sharp enough for estimating the probability $\P\{S_n\in E\}$ for   infinite sets of integers  $E$,  already in the simple case where $S_n$ is a sum of $n$  independent standard Bernoulli random variables and $E$ an arithmetic progression. \end{abstract}

%%%%%%%%

\maketitle

% \baselineskip= 3pt

%%%%%%%%%%%%%%%%

\def\ddate {\sevenrm \ifcase\month\or January\or
February\or March\or April\or May\or June\or July\or
August\or September\or October\or November\or December\fi\! {\the\day}, \!{\sevenrm\the\year}}

%% Classification and key words; note that the 2010 classification is used:

\renewcommand{\thefootnote}{} {{
\footnote{  \par  \sevenrm{[AS.LLT]1} \ddate{}}
 \renewcommand{\thefootnote}{\arabic{footnote}}
\setcounter{footnote}{0}

\section{Introduction.}\label{s1}
This work is devoted to the study of the local limit theorem and of its recent developments, in the context of some standard random models used in arithmetical number theory.  
%as well as to the almost sure local limit theorem. 
It is also somehow completing the recent paper \cite{GW}. We will be mainly interested in studying the local limit theorem for weighted sums of Bernoulli  variables. As it will be clarified soon, this   turns up to be a fundamental question  in the  local limit theorem theory.
 We first recall some basic results and the used methods.  The local limit theorem was established  already three centuries ago  in the binomial case by De Moivre and Laplace around 1730. Based on Stirling approximation formula of $n!$, it is   a very precise result for moderate deviations.   \begin{lemma}\label{moivre}  Let
$0<p<1$,
$q=1-p$. Let $X$ be such that
$\P\{X=1\}=p=1-\P\{X=0\}
$. Let
$X_1, X_2,\ldots$ be independent copies of
$X$ and let $S_n=X_1+\ldots +X_n$. Let   $0<\g<1$ and let
 $ \b \le \g\sqrt{  pq}\, n^{1/3} $. Then for all $k$ such that
letting  $ x= \frac{k-np}{\sqrt{  npq}} $, $|x|\le \b n^{1/6}$,
we have
 \begin{eqnarray*} \P\{S_n=k\}  &=&   \frac{e^{-  \frac{x^2}{  2  }} }{\sqrt{2\pi npq}} \   e^E ,\end{eqnarray*}
with
 $|E|\le  \frac{ |x|^3}{\sqrt{ npq}}+ \frac{|x|^4 }{npq}+ \frac{|x|^3}{2(npq)^{\frac{3}{2}} } +  \frac{1}{ 4n\min(p,q)(1 -  \g      )}$.
   \end{lemma}
This slightly more precise formulation than the one given in Chow and Teicher \cite{CT}, p. 46,  is easily extrapolated  from their proof. 
More generally, let $  \widetilde
X=\{X_n , n\ge 1\}$ be a sequence of independent, square integrable
random variables taking values in a common lattice $\mathcal L(v_{
0},D )=\{v_{ 0}+D k, k\in \Z\}$, where
 $v_{0} $ and $D >0$ are   real numbers.
 % Then $S_n=\sum_{j=1}^nX_j$ takes values in the lattice
%$\mathcal L( v_{ 0}n,D )$. 
Let also $M_n=\sum_{j=1}^n\E X_j$,
$\Sigma_n=\sum_{j=1}^n{\rm Var}( X_j)$. We say that  $\widetilde X$ satisfies a local limit theorem if
 \begin{equation}\label{llt}  \D_n:=  \sup_{N=v_0n+Dk }\Big|   \sqrt{\Sigma_n} \P\{S_n=N\}-{D\over  \sqrt{ 2\pi } }e^{-
{(N-M_n)^2\over  2 \Sigma_n} }\Big| = o(1).
\end{equation}

This 
 fine limit theorem  has connections with Number Theory, see for instance
 % Freiman  \cite{F} and  
  Postnikov  \cite{Po}.  If  $\widetilde X$ is an i.i.d. sequence,  then \eqref{llt} holds if and only if the   \lq\lq {\it span}\rq\rq  $D$ is maximal ($D = \sup\big\{d>0;   \exists a\in
\Z:
\P\{X \in a+d\Z\}=1\big\} $). This is Gnedenko's well-known result, which  is also optimal (Matskyavichyus
\cite{Mat}).  Under stronger integrability conditions, the remainder term can be improved (see \cite{IL} Theorem
4.5.3), \cite{[P]} Theorem 6 p.197). 
%In the Bernoulli case for instance, we have the following result, relevant in %this work.
%\begin{lemma}{\rm (\cite{[P]},  Chapter 7, Theorem 13)}  \label{lltber}Let
%$\mathcal B_n=\b_1+\ldots+\b_n$, $n=1,2,\ldots$  where$\b_i$ are i.i.d.  %Bernoulli r.v.'s ($\P\{\b_i=0\}=\P\{\b_i=1\}=1/2$).  Then,
% \begin{eqnarray*}  \sup_{z}\, \Big|  \P\big\{\mathcal B_n=z\}
% -\sqrt{\frac{2}{\pi n}} e^{-{ (2z-n)^2\over2 n}}\Big|= o ( {1}/{n^{3/2}} ) .
%  \end{eqnarray*}\end{lemma}
%Although optimal, it is  for moderate deviations like $x\sim n^{1/7}$, %considerably less precise than the old    one of De Moivre (case$p=q$). 
The general form of the local limit theorem
(\cite{IL}, Th. 4.2.1) for i.i.d. random variables, states
\begin{theorem}\label{th:G4}
  In order that  for some choice of constants $a_n$ and $b_n$
$$\lim_{n \to \infty}\sup_{N \in \mathcal{L}(v_0n, D)}\Big|\frac{b_n}{\lambda}\P\{S_n=N\}-g\big( \frac{N-a_n}{b_n}\big)\Big|=0, $$
where $g$ is the density of some stable distribution $G$ with exponent $0< \alpha \leq 2$,
 it is necessary and sufficient that
$$ {\rm (i)}\ \
 \frac{S_n-a_n}{b_n} \buildrel{\mathcal D}\over{\Rightarrow}   G   \ \
\hbox{as $n \to \infty$} \qq\qq
    {\rm (ii)}\ \   \hbox{$D $ is maximal}.$$
\end{theorem}
 \vskip 2 pt 

There are essentially two  approaches used: the method of characteristic functions and the  Bernoulli part extraction method. In the later case, this method is called the        extraction method of the Bernoulli part  of a random
variable and  was developed by McDonald 
\cite{M},
 for proving   local limit theorems in presence of the central limit theorem.     Kolmogorov \cite{K}  (see also
Kolmogorov's interesting comment p.~29) initiated twenty years before a similar approach  in the study of L\'evy's concentration function.
We also mention Arratia,  Barbour and  Tavar\'e \cite{ABT1,ABT}  probabilistic approach in  the study of  the asymptotic behaviour of logarithmic combinatorial structures, and the recent work of  R\"ollin and Ross \cite{R} based on  Landau-Kolmogorov inequalities.
\vskip 2 pt 
An important problem inside the general study of the local limit theorem concerns the case when the considered sums are weighted sums of Bernoulli  variables, the "simple" case when the weights are increasing covering already non-trivial examples of random models used in number theory. The purpose of the next Section is to underline this in providing a few   examples of such models, 
which we believe, are challenging problems for probabilists. 

 Additionally,  for weighted sums of independent or i.i.d. random variables, the  Bernoulli part extraction method reduces the problem to the case of weighted sums of Bernoulli  variables,  thereby  making this case crucial too for the application of this method. 
 
\vskip 2 pt 
The goal of this work is to investigate the local limit theorem for weighted sums of Bernoulli  variables.
In Section \ref{s3}, we use the characteristic function method to prove new  local limit theorems. Next in Section \ref{s4}, we give an application of the almost sure local limit theorem to Burr's  representation problem in additive number theory, using an appropriate random model. Finally, we also give an example showing that  the standard form \eqref{llt} of the local limit theorem, fails to be sharp enough for estimating the probability $\P\{S_n\in E\}$ for   infinite sets of integers  $E$; and this   already in the  simple case where $S_n$ is a sum of $n$  independent standard Bernoulli random variables and $E$ an arithmetic progression. 
%%%%%%%%%%%%%%%
%%%%%%%%%%%%%%%

%%%%%%%%%%%%%%%%%%%%%%%%%%%%%%%%%%%%%%%%%%%%%%%%%%%%%%%%%%%%%%%%%%%%%%%%%%%%%%%%%%%%%%%%%%%%%%%%%%%%%%%%%%%%%%%%%%%%%%%%%%%%%%

%%%%%%%%%%%%%%%
%%%%%%%%%%%%%%%

\section{Some Random Models in Number Theory.}\label{s2}
%%%%%%%%%%%%%%%
%%%%%%%%%%%%%%
%%%%%%%%%%%
   %%%%%%%%%%%
    \subsection{\it{A Probabilistic Model for the Dickman Function.}}  This function
originates from the study by Dickman of the asymptotic distribution of the 
  largest prime factor  $P^+(n)$  of a natural integer $n$.  He has shown that the limit 
\begin{equation} \label{difu}\lim_{n\to \infty} \frac{1}{n}\#\big\{k; 1\le k\le n : P^+(k)\le n^{1/u}\big\}= \rho(u)
\end{equation}
exists, and $\rho(u)$, called the Dickman Function, is defined as the continuous solution of the differential-difference equation
$$u\rho'(u) + \rho(u-1)=0, \qq \quad (u>1)  $$ 
with the initial condition $\rho(u)=1$ for $0\le u\le 1$. We have 
$\int_0^\infty \rho(v)\dd v=e^\gamma$, where $\gamma$ is Euler's constant.  This is a function of first importance in analytic number theory,
which has been thoroughly investigated by  
Hensley, Hildebrand, Tenenbaum notably, see \cite{T}   for more details.  
 
There is a probabilistic way of describing the Dickman Function. We refer to       Hwang and  Tsai \cite{HT}.  Let $X=\{X_j, j\ge 1\}$ be a sequence of independent random
variables such that
\begin{equation}\label{dipro}\begin{cases}\P\{ X_j=j\}= j^{-1}\cr 
\P\{ X_j=0\}= 1-j^{-1}.
\end{cases}\qq (j\ge 1)\end{equation}
\begin{proposition}\label{dipropo}Let $D_n= \sum_{j=1}^n X_j$. Then
$$\lim_{n\to \infty} \P\big\{ n^{-1}D_n<x\big\}= e^{-\gamma}\int_{0}^x \rho(v) \dd v \qq \quad (x>0). $$
  \end{proposition}
 
Arratia, Barbour and Tavar\'e 
\cite{ABT}, Corollary 2.8 proved a (restricted)  local limit theorem for $D_n$
\begin{equation}\label{dilo}\lim_{n\to \infty} n\P\{D_n=k_n\}= e^{-\gamma}\rho(x), \qq \hbox{when  $\lim_{n\to \infty} k_n/n= x>0$.}
\end{equation}
The almost sure local limit theorem was recently established in   Giuliano, Szewczak and Weber in \cite{GSW}. The proof is essentially based on a long and delicate study of the related correlations functions. A proof of the local limit theorem in the form (\ref{dilo}) using only characteristic functions is also given, correcting the one indicated \cite{HT}. 
No local limit theorem for $X$ (in the sense of \eqref{llt}) is known.
 \begin{remarks}\label{rm1} \rm  (i) Hensley \cite{H} has shown that the limiting law is infinitely divisible. In the same paper, he also constructed another very interesting probabilistic
model, adapted to the "psixiology" i.e. to functions $\Psi, \Phi$ linked to $P^+, P^-$.
\vskip 1 pt
\noi (ii) Obviously $D_n$ also reads  as $D_n=\sum_{j=1}^n  j\b_j$ where  
$\{\b_j, j=1,\ldots n\}$ are independent Bernoulli random variables such that \begin{equation}\label{dipro}\begin{cases}\P\{ \b_j=1\}= j^
{-1}\cr 
\P\{ \b_j=0\}= 1-j^{-1}.\end{cases}\qq (j=1,\ldots , n)\end{equation}
 \noi (iii)  Let $Z_1,
\ldots Z_n$  be independent Poisson distributed random variables with intensity  $\E Z_j= 1/j$, and let
$T_n=\sum_{j=1}^n  jZ_j$. Then we have the  exact formula 
$\P\{T_n =n\}= e^{-\sum_{j=1}^n 1/j}$, based on Cauchy formula for cycles of permutations (\cite{ABT1}, formula (1.2)).
\vskip 1 pt
\noi    (iv) 
 \label{Ver}  Vervaart has shown that independent Bernoulli random variables
can be embedded into a Poisson process (see \cite{V}, Chapter 4).  
 \end{remarks}

     %%%%%%%%%%%%%%%%%%%%%%%%%%%%%%%%%%%%%%%%%%%%%%%%%%%%%%%%%%%%%%%%%%%%%%
  
  \subsection{\it A Diophantine Equation.} \label{sade}
 Let $\mathcal N=\{\nu_0, \ldots, \nu_P \}$ be a finite set    of integers.
    Consider the
diophantine equation 
 \begin{equation}\label{Deq01}  x_1+ \ldots +  x_n=  y_1+ \ldots +  y_n,  \end{equation}  
  in which the unknown $x_i, y_j$, $1\le i,j\le n$, are subject to belong   to     $\mathcal N$.
   Let $N_n(\mathcal N) $ denote the number  of  
$2n$-uples $(x_1, \ldots, x_n, y_1,
\ldots, y_n)\in   \mathcal N^{2n}  $  which satisfy  (\ref{Deq01}). 
 
\vskip 3pt

    Examine the basic case
 $  \mathcal N=\{0, \ldots ,P-1\}  $ and note $N_n(P)=N_n(\mathcal N)$. Recall the approach used in \cite{Po} \S2.4.  
Let    $X$ be a random variable defined   by
%$$ \P\{ X=k\}={u(k)\over P^2}, \qq \quad k\in \Z .$$
  \begin{equation*}   \P\{X=k\}  =
  \begin{cases} {P-|k|\over P^2}         &  \qq {\rm if} \ 0\le |k|<P, \cr   
 0          &   \qq {\rm if} \ |k|\ge P.  \end{cases}     \end{equation*} 
  We easily verify that $\E X= 0$,  $\s^2=\E X^2= \frac{ P^2-1   }{6}$ and $\E |X|^3\le CP^3$. Moreover, $\E e^{2i\pi tX}= (1/P)  F_{P-1}(2\pi t)   $ where $F_m$ is the Fej\'er kernel,
$$  F_m(u) ={1\over m+1} \Big({\sin   \frac{m+1}{2}u\over
\sin \frac{u}{2}}\Big)^2 . $$
Note that if  $u(k)$
is  the number of solutions of the equation    $x-y=k$, $0\le x\le P-1$, 
$0\le y\le P-1$, then   $u(k)= P-|k|$ if $|k|< P$, and $u(k)=0$ if $|k|\ge P$.  So that in turn $ \P\{ X=k\}={u(k)\over P^2}$, $k\in \Z$. Let $X_1,\ldots, X_n$ be independent copies of $X$ and note $S_n= X_1+\ldots +X_n$. As   $ (x_1-y_1)+ \ldots +
(x_n-y_n)=0$ if and only if  
$  x_1-y_1=k_1, 
\ldots ,x_n-y_n=k_n$, for some integers $k_j$ verifying $ k_1 + \ldots + k_n =0$, we have
 \begin{eqnarray*}\label{sol1}\P\{ S_n=0\}
 &=& 
 \sum_{k_1+\ldots +k_n=0\atop |k_i|< P} \P\big\{X_1=k_1\big\}\ldots  \P\big\{X_n=k_n\big\} \ =\
{N_n(P)\over  P ^{2n}}.
\end{eqnarray*} 
    We have, as  a direct consequence of the  approximate local limit theorem with effective remainder given in \cite{GW}, Corollary 1.8,
\begin{equation}\label{ts01}     
   { N_n(P)\sqrt n\over  P^{2n-1} } =   \sqrt {3/\pi} +\mathcal O\Big(\frac{1}{P^2} +\frac{P   }{ \sqrt  n 
      }   \Big)  ,\end{equation}
uniformly over $n,P$ such that for  $n\ge
C P^2$. 
 
\begin{remark} (i) As  $ \P\{S_n=0\}=\int_{0}^1   |{\sin P\pi t\over
P\sin \pi t}|^{2n}\dd t$,  it is easy to bound
from below   
$N_n(P)$ by $   C { P ^{2n-1}/\sqrt{ n  } }$  and to get the upper bound  $   C_\e { P ^{2n-1+\e}/\sqrt{ n  } }$, for any $\e>0$, uniformly in $P$
and $n$.   
See for instance  \cite{W4}, inequality (2.3).  
\end{remark}  In fact, one \lq\lq can\rq\rq take $\e=0$. 
\begin{theorem}[\cite{AGW}, Th. 2.1] \label{agw}There exist  absolute constants $C',C''$
such that for  any positive integers
$P$ and
$n$,
\begin{equation}       C'\, { P^{2n-1}\over  \sqrt n}\le  N_n(P) \le  C''\, { P ^{2n-1}\over\sqrt{ n  } } .
\end{equation}
\end{theorem}
 The proof   depends on  finer bounds of the previous Fej\'er integrals,  requiring more elaborated calculations.   
 
\begin{remark} We don't exactly know how  the normalized ratios
 $ { N_n(P)\over  P^{2n} }$ behave when $n$ and $P$ vary simultaneously; a question which  is tightly related to the variation properties  of powers of the Fej\'er kernels $\{F^n_{P_j}(u) , j\ge 1\}$ for growing sequences $\{P_j, j\ge 1\}$.
\end{remark} 
   %%%%%%%%%%%%%%%%%%%%%%%%%%%%%%%%%%%%%%%%%%%%%%%%%%%%%%%%%%%%%%%%%%%%%%%%%%%%%%%%%%%
   
\subsection{\it  Freiman-Pitman's Probabilistic Model of  the Partition Function.}\label{sfp}
This is probably the most informative example. Let $q_m(n)$, $m\le n$,  denote 
  the number of partitions of $n$ into distinct parts, each of which is at least $m$, namely the number of ways to express $n$ as 
\begin{equation} n= i_1+ \ldots +i_r, \qq \qq m\le i_1<\ldots <i_r\le n.\end{equation}
 Let $X_m, \ldots , X_n$ be independent random variables defined by 
\begin{equation}\label{pvfp} \P\{X_j =0\}= \frac{1}{1+e^{-\s j}}, \qq\qq \P\{X_j =j\}= \frac{e^{-\s j}}{1+e^{-\s j}}.\end{equation}
  The random variable $Y= X_m + \ldots + X_n$ can serve to modelize the partition function  $q_m(n)$.  There is a one-to-one correspondence between the number of partitions of $n$ of the required type and the number of vectors $(x_m, \ldots, x_n)$ with $x_j=0$ or $1$ such that 
$ m x_m+\ldots + nx_n=n $.
Notice that 
\begin{eqnarray*}    & &e^{\s n}  \int_0^1 \prod_{j=m}^n \big(1 + e^{-\s j} e^{2i\pi \a j}\big)e^{-2i\pi \a n} \dd  \a \cr 
&=&  e^{\s n}\sum_{  x_j\in\{0,1\}\atop m\le j \le n}  \int_0^1   e^{-\s (m x_m+\ldots + nx_n)}   e^{-2i\pi(m x_m+\ldots + nx_n-n) \a  } \dd  \a 
\cr 
&=&  e^{\s n}\sum_{  x_j\in\{0,1\}\atop m\le j \le n}     e^{-\s n}   \chi\{m x_m+\ldots + nx_n=n\}
\ = \   q_m(n).  \end{eqnarray*}
Hence the  formula (in which $\s$   only appears in the right-hand side)
\begin{equation}   \label{qmn} q_m(n) = e^{\s n} \int_0^1 \prod_{j=m}^n \big(1 + e^{-\s j} e^{2i\pi \a j}\big)e^{-2i\pi \a n} \dd  \a .  \end{equation}
   This also implies (letting $\p(t) =\E e^{2i\pi tY}$ be the characteristic function of $Y$)
\begin{eqnarray*}    q_m(n)& = &e^{\s n}\Big( \prod_{j=m}^n  ( {1+e^{-\s j}})\Big)\int_0^1 \prod_{j=m}^n \Big(\frac{1 + e^{-\s j} e^{2i\pi \a j}}{1+e^{-\s j}}\Big)e^{-2i\pi \a n} \dd  \a 
%\cr &=&  e^{\s n}\Big( \prod_{j=m}^n  ( {1+e^{-\s j}})\Big)\int_0^1 \prod_{j=m}^n  
%\p_j(\a)e^{-2i\pi \a n} \dd  \a
\cr &=&  e^{\s n}\Big( \prod_{j=m}^n  ( {1+e^{-\s j}})\Big)\int_0^1    \p (\a)e^{-2i\pi \a n} \dd  \a
\ = \  e^{\s n}\Big( \prod_{j=m}^n  ( {1+e^{-\s j}})\Big)\P\{Y=n\} .\end{eqnarray*}
    In \cite{FP} p.\,387 and 389, the authors noticed that  an appropriate local limit theorem 
would  allow to write $ \P\{Y=n\} \sim   e^{-(\E Y-n)^2/(2{\rm Var}(Y))}/{\sqrt{2\pi{\rm Var}(Y)}}$. 
 Choosing $\s$ as being the unique solution of the equation $\E Y=\sum_{j=m}^n \frac{j}{1+e^{\s j}}= n$
would then give 
$  \P\{Y=n\} \sim {1}/{B\sqrt{2\pi }} $, and by reporting 
\begin{eqnarray*}    q_m(n) \sim   
  e^{\s n}\Big( \prod_{j=m}^n  ( {1+e^{-\s j}})\Big) \frac{1}{B\sqrt{2\pi }} .  \end{eqnarray*}
In place,  Freiman and Pitman   directly estimated the integral  in        \eqref{qmn}   in a 
 long delicate  work \cite{FP}. 
 \begin{remark}  
 By Euler's pentagonal theorem, $q_0(n)$   appears as  a coefficient in the expansion of $\prod_{k\le n}(1+e^{ik\theta})$.  
\end{remark} 
 \subsection{} The basic problem illustrated by the previous examples states as follows.  
\begin{problem}\label{p1}Let $\{k_j, j\ge 1\}$ be  an increasing sequence of positive integers and $\{p_j, j\ge 1\}$ be a sequence of reals in $]0,1[$.  Describe the CLT
and LLT for the sequence
$S_n= k_1\b_1+
\ldots + k_n\b_n$, $n\ge 1$,  where 
$\b_j$ are independent Bernoulli random variables defined by 
\begin{equation}\label{dipro}\begin{cases}\P\{ \b_j=1\}= p_j\cr 
\P\{ \b_j=0\}= 1- p_j.
\end{cases}\qq (j\ge 1)\end{equation}
\end{problem}
In the Freiman-Pitman
model, the system of independent random variables varies with the choice of the integer. And so there is, properly speaking, no
central limit theorem involved  and thereby no local limit theorem either, except when  placing the problem in the setting
of triangular arrays. Corresponding forms of the central limit theorem exist. As to suitable versions of the local limit theorem for
triangular arrays  {\it with} remainder term, we don't know whether such a result exists in the litterature. Thus it makes sense to also consider a "local" version of the previous problem. 
\begin{problem}[Finite version]\label{p2}
 To obtain effective sharp   estimates of 
\begin{equation*} \P\{ S_n = N\}.\end{equation*}
 We refer to \cite{GW} where this question is  investigated. \end{problem} 
Returning to  the Freiman-Pitman
model, we   observe that the relevant question  rather concerns the search of   sharp  estimates of 
$\P\{S_n=0\}$ (namely of $\int_{-1/2}^{1/2}\E e^{2i \pi  t S_n}\dd t$), the random variables being centered, than working out a local limit theorem, which
is quite another problem. Nevertheless, this model, as well as  others previously reviewed, sheds light on   limitations 
 to the domain of validity of the local limit
theorem, in a quite informative way. 

Some further useful remarks are necessary. We note throughout  $\{\varsigma, \varsigma_j, j\ge 1\}$ a sequence independent standard Bernoulli random variables (namely associated with $p_j\equiv 1/2$) and  $$ T_n = \varsigma_1+
\ldots + \varsigma_n\qq\quad n\ge 1.$$
 
  \begin{remark}[Reduction to standard Bernoulli random variables] Let  $\b $  be a Bernoulli random variable with  $\P\{\b  =1\}= \a = 1-\P\{\b  =0\}$. Assume $0<\a<1/2$.  Let $\e, \varsigma$    be such that $\b, \e, \varsigma$ are independent and  $\P\{\e  =1\}= 2\a = 1-\P\{\e  =0\}$. Trivially  $\e\varsigma \buildrel{\mathcal L}\over{=}\b$. We can thus write when $0<p_j<1/2$, $1\le j\le n$,
 $$S_n= k_1\e_1\varsigma_1+
\ldots + k_n\e_n\varsigma_n.$$
 Problem \ref{p2} reduces to first estimate (conditionnally to $\e_j$) a sum of the same kind
 $$ T'_n = k'_1\varsigma_1+
\ldots + k'_n\varsigma_n$$
with $k_j'$ increasing, but where the Bernoulli random variables are standard. 
\end{remark}
\begin{remark} If $1/2<a<1$, let $\tau_0 $
be  verifying
$0<  \tau_0< 2\min (\a,1-\a)$. 
Define a pair of random variables
$(V,\e)$   as follows.  
\begin{eqnarray*}   \begin{cases} \P\{ (V,\e)=(1,1)\}=0  
      \cr   \P\{ (V,\e)=(1,0)\}=\a -{\tau_0 \over
2}    . \end{cases}  \qq \begin{cases}   \P\{ (V,\e)=(0,1)\}=\tau_{0 } 
      \cr 
 \P\{ (V,\e)=( 0,0)\}=1-\a -{ \tau_0\over
2}    . \end{cases} 
\end{eqnarray*}
  Let $ \varsigma$ be  independent from $(V,\e)$.     Then 
$V+\e  \varsigma\buildrel{\mathcal L}\over{=}\b$. 
\end{remark}

 \begin{remark} Fix the integer $n$. Let $ T_m' = k_m\varsigma_m+
\ldots + k_n\varsigma_n$,  $1\le m\le n$, and consider the parallelogram   $H_m=\big\{ h=k_{i_1}+ \ldots +k_{i_r} \hbox{with $i_1<\ldots <i_r$ and $1\le r\le m$}\big\}.$
Then we have the following formula \begin{eqnarray}\label{progr.def}\P\{T'_1= b\}& = & \frac{1}{2^m} \sum_{h\in H_m\cup\{0\}}\P\{ T'_{n-m}= b-h\} .\end{eqnarray}

\end{remark}
 %%%%%%%%%%%%%%%%
 %%%%%%%%%%%%%%%%
 
\section{Weighted Local Limit Theorems.}\label{s3}

 %%%%%%%%%%%%%%%%
 %%%%%%%%%%%%%%%%
 %%%%%%%%%%%%%%%%

We use the characteristic function method to study the   local limit theorem for the sums 
$$B_\nu=
\b_1+\ldots +\b_\nu, \qq \quad \nu=1, 2,\ldots$$
where  $\b_j$ are independent random variables  defined by
\begin{eqnarray}\label{betai}\P\{
\b_j=0\}= \t_j,\qq 
\P\{
\b_j=k_j\}  =1-\t_j,
\end{eqnarray}
  with $0<\t_j<1$ for each $j$, and $k_j$ are increasing positive weights.  Let  $$  {\rm Var}({B_\nu})= \sum_{j=1}^\nu  (1-\t_j) \t_j k^2_j , \qquad \t = \inf_{j=1}^\nu   \t_j (1-\t_j).$$
 
   \begin{theorem}\label{wfp}Let $\varrho n\le k<k+\nu \le n $  
where $0<\varrho<1$, $n$ is some positive integer, and let    $k_j=k+ j-1$, $j=1, \ldots, \nu$. Let $ { 1/24}<\e< { 1/6}   $. For every $m\in \Z$, 
  \begin{eqnarray*}  \bigg|\P\{B_\nu=n\} -\frac{e^{ -   \frac{      (  \sum_{j=k}^{k+\nu-1} (1-\t_j)j-n)^2}{ 2 {\rm Var}({B_\nu})} 
}}{\sqrt{2\pi {\rm Var}({B_\nu})}}   \bigg|
 &\le  & \frac{C } {   \sqrt{   {\rm Var}  ({B_\nu})  
}      }\Big(\frac{\nu^{1/6-4\e}}{\t ^{1/3}\rho   } + 
 {  e^{      -  2\pi^2\t \rho ^2 \d^2\nu^{1/3}}     }\Big)
  . 
\end{eqnarray*}
\end{theorem}
 
  Put
\begin{equation}\label{PMB}P_\nu =\sum_{j=1}^\nu  (1-\t_j)    , \qq M_\nu =\sum_{j=1}^\nu  (1-\t_j)   k_j , \qq B_\nu=   \sum_{j=1}^\nu 
(1-\t_j)
\t_j k^2_j.
\end{equation} 
\begin{theorem}\label{lltweibin}\begin{eqnarray*}
\sup_{n\in \Z}\bigg|\P\{B_\nu= n\} -\frac{1}{\sqrt{2\pi B_\nu}} e^{ -   \frac{   ( n-M_\nu)^2}{ 2 B_\nu}  }  \bigg| 
&\le &         \frac{C}{P_\nu}. \end{eqnarray*}
 Further  if    $n_0:=\sum_{j=1}^\nu (1- \t_j)k_j   $ is integer, then 
\begin{eqnarray*}  \Big| \P\{\b_1+\ldots +\b_\nu=n_0\}-\frac{1 }{\sqrt{ 2\pi B_\nu}}\Big|  &\le  & \frac{C } {P_\nu } .\end{eqnarray*}
\end{theorem}
%\begin{remark} It is clear that the domain of applicability is limited by the constraint
%$$\frac{1 } {P_\nu }\ll \frac{e^{ -   \frac{   ( n-M_\nu)^2}{ 2 B_\nu}  }}{\sqrt{2\pi B_\nu}} .$$
%\end{remark}
Before passing to the proofs, we begin with making a brief analysis. Let $\p_j(t) =\E e^{2i\pi t\b_j}$, 
$\p_{B_\nu}(t)=\E e^{2i\pi
t{B_\nu}}$. By the Fourier inversion formula, 
\begin{eqnarray*}\P\{{B_\nu}=m\}= \int_{|t|\le \tau}e^{-2i \pi m t} \p_{B_\nu}(t) \dd t
+ \int_{\tau \le |t| \le \frac{1}{2}}e^{-2i \pi m t} \p_{B_\nu}(t) \dd t:= I_{\tau}(\nu,m)+I^{\tau}(\nu,m),
\end{eqnarray*}  
where $\tau>0$ will be chosen to be small,  depending of $m$.  The first integral term produces the main term and is easily tractable. The estimation of the second integral term  is in fact the hard part of the problem, where all the difficulty is concentrated. It is necessary to show that
$$ \Big|\int_{\tau \le |t| \le \frac{1}{2}}e^{-2i \pi m t} \p_{B_\nu}(t) \dd t\Big| \ll \frac{1}{\sqrt{ {\rm Var}({B_\nu})}}.$$
%For, as is made for instance in \cite{FP}, 
There seems to be no other way than controlling 
$ \int_{\tau \le |t|\le \frac{1}{2}}  | \p_{B_\nu}  (t) |\dd t. $
%As $\tau$ is close to $0$, this amounts to control (up to a small error factor)
%$ \int_{- {1}/{2}}^{ {1}/{2}}  | \p_{B_\nu}  (t) |\dd t$. 
From Lemma \ref{lfp1}-(i) will follow that 
$|\p_{B_\nu} (t)|\le \exp\big\{- 2\sum_{j=1}^\nu \t_j(1- \t_j)\sin^2\pi t k_j\big\}$.
The whole matter consequently directly depends on the behaviour of the sine sum
 $$\sum_{j=1}^\nu \t_j(1- \t_j)\sin^2\pi t k_j$$
 away from $0$,  an  obviously difficult question. Thus,    answers can be expected only for specific cases.  
%%%%%%%%%%%
%%%%%%%%%%%
\subsection{Estimates of $\boldsymbol{I_{\tau}(\nu,m)}$.}  

 Recall Lemma 3 in \cite{FP}. Although stated with the choice of probability values
given by (\ref{pvfp}),  this lemma is general. For completion, we have included a slightly shorter  proof. 
 \begin{lemma}\label{lfp1} Let $m  $ be a positive real and $p$ be a real such that $0<p<1$. Let $\b$ be a random variable defined by $\P\{ \b=0\}= p$, $
\P\{
\b=m\}= 1-p=q$.  Let $\p(t) =\E e^{2i\pi t \b}$. Then 
 we have the following estimates,
\vskip 3 pt\noi\ \  {\rm (i)} For all real $t$,
 $|\p (t)|\le \exp\big\{- 2pq\sin^2\pi t m\big\} $ \vskip 3 pt
\noi\ \  {\rm (ii)} If $q|\sin \pi t m|\le 1/3$,
  \begin{equation*}
  %\label{fp2}
  \p (t)= \exp\big\{ 2i \pi    qmt - 2\pi^2pqm^2 t^2  + B(t) \big\},
\end{equation*}
and $|B(t)|\le C   qm^3t^3  $,  the constant $C$ being absolute. \end{lemma}
 
\begin{proof}   One verifies that $|\p(t)|^2=1- 4 pq \sin^2\pi mt.$ As moreover $1-\t\le e^{-\t}$ if  $\t \ge 0$, (i) follows. 
  Write now $\p(t)= 1 +
q\big( e^{2i\pi mt} -1\big)=1 + u$ and notice that $|u|= 2q|\sin\pi mt|$. We use the fact that   if $|\theta|\le 2/3$, then 
$$ 1+\theta= \exp\{ \theta -\theta^2 +B\}, \qq \quad |B|\le C |\theta|^3.$$
And $C$ is an  absolute   constant. 
 From the bound $| e^z-( 1+ {z\over 1!}+\ldots
+{z^n\over n!})|\le {|z|^{n+1}\over (n+1)!}e^{|z|}$, valid for $z\in \C$    and   $n\in\N_*$ (\cite{Mi}, 3.8.25),  we get by applying it with  $z=2i\pi
mt$,  
\begin{eqnarray*}\big| u-q\big( 2i\pi mt -2 (\pi mt)^2\big)|&\le&    C qm^2 | t|^2  \cr 
\big|u^2+(2q  \pi m t)^2  | &\le& Cq^2m^3| t|^3   
 \cr 
|u|^3&\le & Cq^3m^3| t|^3 .
\end{eqnarray*}
As we assumed $q|\sin \pi t m|\le 1/3$, we consequently  find that
$$\p(t) =  1+u= \exp\{ u -u^2 +B\}=\exp\big\{ 2i \pi    qmt - 2\pi^2pqm^2 t^2  + B(t) \big\} 
      ,$$
with $|B(t)|\le C   qm^3t^3 $. 
  \end{proof}

  The next Lemma provides an estimate for the main integral term.  Let $0<\d\le \frac{1}{3\pi}$ and put
\begin{equation}\label{taud}  \tau = \frac{\d   }{( \sum_{j=1}^\nu (1- \t_j)k^3_j)^{1/3}} 
% \qquad L= \frac{( \sum_{j=1}^\nu (1-\t_j)k^3_j)^{1/3}} { \big(\sum_{j=1}^%\nu (1- \t_j)k^2_j\big)^{ 1/2}},
  .
\end{equation}
   \begin{lemma} \label{lltb} 
For every $n\in \Z$, 
\begin{eqnarray*}
 \bigg|\int_{-\tau}^\tau  e^{- 2i \pi   nt} \p_{B_\nu}  (t)\dd t -\frac{e^{ -   \frac{      (  \sum_{j=1}^\nu (1- \t_j)k_j-n)^2}{ 2 {\rm Var}({B_\nu})} 
}}{\sqrt{2\pi {\rm Var}({B_\nu})}}   \bigg|
 &\le  &  C\Big( \tau \d^3 +    \frac{  e^{      -  2\pi^2 \tau^2{ \rm Var}({B_\nu}) }     } {   \sqrt{   {\rm Var}  ({B_\nu})   }      }  \Big)  
  . 
\end{eqnarray*} 
 Further  if    $n_0:=\sum_{j=1}^\nu (1- \t_j)k_j   $ is integer, then 
  \begin{eqnarray*} \int_{-\tau}^\tau  e^{- 2i \pi   n_0t}  \p_{B_\nu}  (t) \dd t &= & \frac{1 }{\sqrt{ 2\pi{\rm Var}({B_\nu})}}\, \big(1+  B \big) ,
\end{eqnarray*} with 
$$|B|\le C \Big(\, \tau \d^3+\frac{e^{      - {2\pi^2\tau^2{\rm Var}({B_\nu}) }}}{1+  \sqrt{   {\rm Var} ({B_\nu})    }      }\, \Big). $$
 \end{lemma}

\begin{proof}[Proof of Lemma \ref{lltb}]
 As $\d\le \frac{1}{3\pi}$, we observe that for $j=1,\ldots, \nu$, $$  \sup_{|t|\le \tau} (1- \t_j)|\sin \pi tk_j| \le (1- \t_j) \pi \tau k_j = \frac{\d (1- \t_j) \pi k_j  }{( \sum_{j=1}^\nu (1- \t_j)k^3_j)^{1/3}}\le  \d (1- \t_j)^{1/3} \pi   \le \frac{1}{3}.$$
      Lemma \ref{lfp1} thus  implies, $$ \p_{B_\nu}(t) =\prod_{j=1}^\nu \p_j(t)= 
\exp\Big\{ 2i \pi     t\sum_{j=1}^\nu (1- \t_j)k_j - 2\pi^2t^2{\rm Var}({B_\nu})   + B_1(t)
\Big\},
$$
and $|B_1(t)|\le C |t|^3\sum_{j=1}^\nu (1- \t_j)k^3_j   $. 
\vskip 2 pt
 By  (\ref{taud}), 
     $$\sup_{|t|\le \tau}|B_1(t)|\le C\tau^3\sum_{j=1}^\nu (1- \t_j)k^3_j=C\d^3 .  $$ 
 Noting  then $\varsigma = \sum_{j=1}^\nu (1- \t_j)k_j-n  $ and writing that
 $$ \p_{B_\nu}(t) = 
\exp\big\{ 2i \pi     t(\varsigma+n) - 2\pi^2t^2{\rm Var}({B_\nu})   + B_1(t)
\big\},
$$
we thus deduce  the following  bound \begin{eqnarray*} 
& & \Big|\int_{-\tau}^\tau   \big\{ e^{- 2i \pi   nt}\p_{B_\nu}  (t) - e^{ 2i \pi     t \varsigma -
2\pi^2t^2{\rm Var}({B_\nu})}
   \big\}\dd t\Big|
\cr  
&\le  & \int_{-\tau}^\tau   \big| e^{- 2i \pi   nt} \p_{B_\nu}  (t) - e^{ 2i \pi     t \varsigma -
2\pi^2t^2{\rm Var}({B_\nu})}
   \big|\dd t 
\cr  
&=  & \int_{-\tau}^\tau   \big|  e^{ 2i \pi     t \varsigma -
2\pi^2t^2{\rm Var}({B_\nu})}
  \big( e^{B(t) }-1\big) \big|\dd t 
\cr& \le  & \int_{-\tau}^\tau   e^{  -
2\pi^2t^2{\rm Var}({B_\nu})}\big|  
    e^{B(t) }-1\big|   \dd t \le  \int_{-\tau}^\tau   |  
     B(t)  |   \dd t \le C\tau \d^3.  
\end{eqnarray*}
   Now we also have  that 
\begin{eqnarray*} \int_{-\tau}^\tau e^{ 2i \pi     t \varsigma -
2\pi^2t^2{\rm Var}({B_\nu}) }
   \dd t &=   &  \int_\R  e^{ 2i \pi     t \varsigma -
2\pi^2t^2{\rm Var}({B_\nu}) }  \dd t + H 
 \cr \big(t= \frac{u}{2\pi\sqrt{   {\rm Var}({B_\nu})}}\big)\quad &= & \int_\R e^{    \frac{i       \varsigma u}{\sqrt{   {\rm Var} ({B_\nu}) }}   -
\frac{u^2}{2}  }  \frac{ \dd u}{2\pi \sqrt{ {\rm Var}({B_\nu})}} + H
\cr &= &  \frac{e^{ -   \frac{       \varsigma^2}{ 2 {\rm Var}({B_\nu})}  }  }{\sqrt{ 2\pi{\rm Var}({B_\nu})}}   + H ,
  \end{eqnarray*} 
 and recalling Boyd's estimate \cite[p.~179]{Mi} of Mill's ratio $R(x) =
e^{x^2/2}\int_x^\infty e^{-t^2/2}\ dt$,  $$  {\pi\over \sqrt{x^2+2\pi} +(\pi -1)x}\le R(x) \le
{\pi\over \sqrt{(\pi -2)^2x^2 +2\pi} +2x }   $$
for all $x\ge 0$,
we further have 
\begin{eqnarray*}  | H  |& \le & 
\int_{|u|\ge 2\pi \tau  \sqrt{  {\rm Var}({B_\nu}) }    }   e^{      -
\frac{u^2}{2}  }  \frac{ \dd u}{2\pi\sqrt{ {\rm Var}({B_\nu}) }  } \ =\ \frac{  e^{      -  2\pi^2 \tau^2{ \rm Var}({B_\nu}) }       } {  2\pi   \sqrt{   {\rm Var} ({B_\nu})    }       )   
}\, R\big(2\pi \tau\sqrt{  {\rm Var}({B_\nu}) }\big)
\cr & \le   & C  \  \frac{  e^{      -  2\pi^2 \tau^2{ \rm Var}({B_\nu}) }       } {    \sqrt{   {\rm Var} ({B_\nu})    }  (1+  \sqrt{   {\rm Var} ({B_\nu})    }  )   
} .
\end{eqnarray*}
     
Consequently 
\begin{eqnarray*}
 \bigg|\int_{-\tau}^\tau  e^{- 2i \pi   nt} \p_{B_\nu}  (t)\dd t -\frac{e^{ -   \frac{       \varsigma^2}{ 2 {\rm Var}({B_\nu})}  }}{\sqrt{2\pi {\rm Var}({B_\nu})}}   \bigg|
 &\le  &  C\Big( \tau \d^3 +    \frac{  e^{      -  2\pi^2 \tau^2{ \rm Var}({B_\nu}) }     } {   \sqrt{   {\rm Var}  ({B_\nu})   }  (1+  \sqrt{   {\rm
Var} ({B_\nu})    }  \,)  }  \Big)  
  . 
\end{eqnarray*}

\vskip 3 pt Now if there is an integer $n$   such that $\sum_{j=1}^\nu (1- \t_j)k_j=n $, then 
$\varsigma = 0$ and 
 $ \p_{B_\nu}(t) = 
\exp\big\{ 2i \pi     tn - 2\pi^2t^2{\rm Var}({B_\nu})   + B_1(t)
\big\}$; whence 
\begin{eqnarray}\label{mult}
\int_{-\tau}^\tau  e^{- 2i \pi   nt}  \p_{B_\nu}  (t) \dd t&=& \int_{-\tau}^\tau  e^{  -
2\pi^2t^2{\rm Var}({B_\nu})}\big(1+ e^{ B_1(t)}-1\big)
  \dd t 
\cr 
&= & \big(1+ B\big) \int_{-\tau}^\tau  e^{  -
2\pi^2t^2{\rm Var}({B_\nu})} 
  \dd t  ,  
\end{eqnarray}
with  $|B|\le C\tau \d^3$. As moreover
 $$ \Big| \int_{-\tau}^\tau e^{  -
2\pi^2t^2{\rm Var}({B_\nu}) }
   \dd t - \frac{1 }{\sqrt{ 2\pi{\rm Var}({B_\nu})}} \Big|\le C \frac{  e^{      -  2\pi^2 \tau^2{ \rm Var}({B_\nu}) }  } 
{ \sqrt{ {\rm Var} ({B_\nu}) } (1+  \sqrt{   {\rm Var} ({B_\nu})    }\, ) } , $$
 we have 
 \begin{eqnarray*} \int_{-\tau}^\tau  e^{- 2i \pi   nt}  \p_{B_\nu}  (t) \dd t &= & \frac{1 }{\sqrt{ 2\pi{\rm Var}({B_\nu})}}\big(1+ B\big)\big(1 +B_1 \big) , \end{eqnarray*}
where $|B_1|\le C e^{      - {2\pi^2\tau^2{\rm Var}({B_\nu})}  }/(1+  \sqrt{   {\rm Var} ({B_\nu})    }\,) $. 
We conclude to 
  \begin{eqnarray*} \int_{-\tau}^\tau  e^{- 2i \pi   nt}  \p_{B_\nu}  (t) \dd t &= & \frac{1 }{\sqrt{ 2\pi{\rm Var}({B_\nu})}}\, \big(1+  B_2 \big) ,
\end{eqnarray*} with 
$$|B_2|\le C \Big(\, \tau \d^3+\frac{e^{      - {2\pi^2\tau^2{\rm Var}({B_\nu}) }}}{1+  \sqrt{   {\rm Var} ({B_\nu})    }      }\, \Big). $$
   \end{proof}
%%%%%%%%%%%
%%%%%%%%%%%
\subsection{Estimates of $\boldsymbol{I^{\tau}(\nu,m)}$.}  
We assume here that $k_j=k+ j-1$, $j=1, \ldots, \nu$. Then,
\begin{lemma}\label{suptau}
$$\int_{|t|>  \tau}|\p_{B_\nu}(t)|\dd t \le    e^{- \frac{\t \nu^3\tau^2}{2} }.$$\end{lemma}   
\begin{proof}This is an immediate consequence of the following lemma \begin{lemma}[\cite{FP}, Lemma 8]\label{lfp2} For $|t|\le 1/2$ and any positive integers $m$ and $k$
such that
$k\ge 2$ we have
$$ \sum_{j=m}^{m+k-1}  \sin^2 \pi j t\ge \frac{k}{4} \min\big( 1, (tk)^2\big).$$
\end{lemma}
By  Lemma \ref{lfp1}-(i), 
$$|\p_{B_\nu} (t)|\le \exp\big\{- 2\sum_{j=1}^\nu \t_j(1- \t_j)\sin^2\pi t k_j\big\}\le  \exp\big\{- 2\t\sum_{j=1}^\nu \sin^2\pi t k_j\big\}.$$
Thus \begin{eqnarray*}\int_{|t|>  \tau}|\p_{B_\nu}(t)|\dd t\ \le \ \int_{|t|>  \tau}e^{-   \frac{\t}{2}\sum_{j=1}^{\nu}\sin^2\pi jt  }
\dd t
\ \le \ \int_{|t|>  \tau}e^{- \frac{\t\nu}{2}  \min(1,   \nu  |t|)^2 } \dd t
\ \le e^{- \frac{\t \nu^3\tau^2}{2} }.\end{eqnarray*} 
\end{proof}

 %%%%%%%%%%%%%%%%%
%%%%%%%%%%%%%%%%%%
\subsection{Proof of Theorem \ref{wfp}.}
%%%%%%%%%%%%%%%%%%

By Lemma \ref{lltb} applied with $k_j=k+ j-1$, $j=1, \ldots, \nu$, for every $n\in \Z$, 
\begin{eqnarray*}
 \bigg|\int_{0}^\tau  e^{- 2i \pi   nt} \p_{B_\nu}  (t)\dd t -\frac{e^{ -   \frac{      (   \sum_{j=k}^{k+\nu-1} (1-\t_j)j-n)^2}{ 2 {\rm Var}({B_\nu})} 
}}{\sqrt{2\pi {\rm Var}({B_\nu})}}   \bigg|
 &\le  &  C\Big( \tau \d^3 +    \frac{  e^{      -  2\pi^2 \tau^2{ \rm Var}({B_\nu}) }     } {   \sqrt{   {\rm Var}  ({B_\nu})   }      }  \Big)  
\cr &\le  & C  \ \frac{  {\nu^{1/6-4\e}}\rho^{-1    } + 
 {  e^{      -  2\pi^2\rho ^2 \d^2\nu^{1/3}}     } } {   \sqrt{   {\rm Var}  ({B_\nu})  
}      }    . 
\end{eqnarray*}

By combining  with Lemma \ref{suptau}
and using Fourier inversion formula,
\begin{eqnarray*}  \bigg|\P\{B_\nu=n\} -\frac{e^{ -   \frac{      (  \sum_{j=k}^{k+\nu-1} (1-\t_j)j-n)^2}{ 2 {\rm Var}({B_\nu})} 
}}{\sqrt{2\pi {\rm Var}({B_\nu})}}   \bigg|
 &\le  &  C\Big( \tau \d^3 +    \frac{  e^{      -  2\pi^2 \tau^2{ \rm Var}({B_\nu}) }     } {   \sqrt{   {\rm Var}  ({B_\nu})   }      } + e^{- \frac{\t\nu^3\tau^2}{8} } \Big)  
  . 
\end{eqnarray*}We have the following estimates
%$$ \frac{\d   }{n\nu^{1/3}}\le  \tau = \frac{\d   }{( \sum_{j=1}^\nu (1- \t_j)%k^3_j)^{1/3}}\le   \frac{\d   }{\rho n( \t  \nu )^{1/3}}$$
%$${\rm Var}({B_\nu})= \sum_{j=1}^\nu  (1-\t_j) \t_j k^2_j $$
%\begin{equation}\label{taud}  \tau = \frac{\d   }{( \sum_{j=1}^\nu (1- \t_j)k^3_j)^{1/3}} , \qq \t = \inf_{j=1}^\nu   \t_j (1- \t_j)  .
%\end{equation}
\begin{equation*}\begin{cases}
{\rm (i)}\quad & \frac{\d   }{n\nu^{1/3}}\le  \tau = \frac{\d   }{( \sum_{j=1}^\nu (1- \t_j)k^3_j)^{1/3}}\le   \frac{\d   }{\rho n( \t  \nu )^{1/3}},
\cr {\rm (ii)}\quad & \t(\rho n)^2\nu\le {\rm Var}({B_\nu})= \sum_{j=1}^\nu  (1-\t_j) \t_j k^2_j \le n^2\nu,
\cr {\rm (iii)}\quad & \tau^2{\rm Var}({B_\nu})\ge  \frac{\t \d^2}{  n^2 \nu^{2/3}}(\rho n)^2\nu= \t\rho ^2 \d^2\nu^{1/3}  .
\end{cases}\end{equation*}
Choose $\d =\nu^{-\e}$ with $ { 1/24}<\e< { 1/6}   $. Then
$$\tau \d^3 \le  \frac{\d^4}{\rho n (\t\nu)^{1/3}}=\frac{1}{\rho n\t^{1/3} \nu^{1/3+ 4\e}}=\frac{\nu^{1/6-4\e}}{\t ^{1/3}\rho n \nu^{1/2}}\le
\frac{\nu^{1/6-4\e}}{\t ^{1/3}\rho\sqrt{{\rm Var}({B_\nu})}  }.
$$
We pass to the control of the error terms. For the major integral term we have,
$$\tau \d^3 +    \frac{  e^{      -  2\pi^2 \tau^2{ \rm Var}({B_\nu}) }     } {   \sqrt{   {\rm Var}  ({B_\nu})   }      }\le \frac{1 } {   \sqrt{   {\rm Var}  ({B_\nu})  
}      }\Big(\frac{\nu^{1/6-4\e}}{\t ^{1/3}\rho   } + 
 {  e^{      -  2\pi^2\t \rho ^2 \d^2\nu^{1/3}}     }\Big).$$
   Consequently, 
   \begin{eqnarray*}  \bigg|\P\{B_\nu=n\} -\frac{e^{ -   \frac{      (  \sum_{j=k}^{k+\nu-1} (1-\t_j)j-n)^2}{ 2 {\rm Var}({B_\nu})} 
}}{\sqrt{2\pi {\rm Var}({B_\nu})}}   \bigg|
 &\le  & \frac{C } {   \sqrt{   {\rm Var}  ({B_\nu})  
}      }\Big(\frac{\nu^{1/6-4\e}}{\t ^{1/3}\rho   } + 
 {  e^{      -  2\pi^2\t \rho ^2 \d^2\nu^{1/3}}     }\Big)
  . 
\end{eqnarray*}

 %%%%%%%
 \subsection{Other Estimates of $\boldsymbol{I^{\tau}(\nu,m)}$.} 
 %%%%%%%
   The following lemma is relevant. Introduce 
for $q\ge 1$ integer,
\begin{equation}\p_{\theta_1, \ldots, \theta_\nu, k_1, \ldots, k_\nu}( q)= \p(q):= \frac{\big\|\sum_{j=1}^\nu\theta_j\cos 2\pi tk_j\big\|_{2q}}{\sum_{j=1}^\nu\theta_j}.
\end{equation}
\begin{lemma}\label{lfp3}  For any $0<c\le 1$,
\begin{eqnarray*}  \int_0^1|\p_{B_\nu}(t)|\dd t  &\le &   \Big(\frac{\p(q)}{c}\Big)^{2q} + e^{-(1-c)\sum_{j=1}^\nu  \t_j(1-\t_j)}  . 
\end{eqnarray*}
\end{lemma}
%  If    $\t_j=1-1/j $, this  yields  for instance by taking $q=1$, 
 %\begin{eqnarray*}\int_0^1|\p_{B_\nu}(t)|\dd t  &\le &    \frac{C}{\log^2 \nu } . \end{eqnarray*} 
 \begin{proof}[Proof of Lemma \ref{lfp3}]   Let  $\theta_j =\t_j(1-\t_j), j= 1\ldots , \nu $ and note 
$ E= \big\{t; |t|\le \frac{1}{2}:\big|\sum_{j=1}^\nu\theta_j\cos 2\pi tk_j\big|>c\sum_{j=1}^\nu 
\theta_j  \big\}$. 
At first, by using Tchebycheff's inequality, 
\begin{eqnarray*}    \l \{E \}  &\le & \big( c\sum_{j=1}^\nu 
\theta_j)^{-2q}\int_{-\frac{1}{2} }^{\frac{1}{2}}\Big|\sum_{j=1}^\nu\theta_j\cos 2\pi tk_j\Big|^{2q}\dd t
=c^{-2q}\p(q)^{2q}  , 
\end{eqnarray*}
Since $|\p_{B_\nu}(t)|\le 1$, we have
\begin{eqnarray*}   \int_0^1|\p_{B_\nu}(t)|\dd t \le   \l \{E \}  + \int_{E^c }|\p_{B_\nu}(t)|\dd t
 \le   c^{-2q}\p(q)^{2q}+ \int_{E^c }|\p_{B_\nu}(t)|\dd t  , 
\end{eqnarray*}
  By Lemma \ref{lfp1}, using that  $2\sin^2 a= 1-\cos 2a$, we have  for all real
$t$, 
\begin{eqnarray*}|\p_{B_\nu}(t)|&=&\prod_{j=1}^\nu|\p_j(t)| \le \exp\big\{-\sum_{j=1}^\nu 2\theta_j\sin^2\pi t k_j\big\}
\cr &= &\exp\big\{-\sum_{j=1}^\nu  \theta_j \big\}\exp\big\{ \sum_{j=1}^\nu  \theta_j\cos2\pi t k_j\big\} .
\end{eqnarray*} 
So that  
\begin{eqnarray*}\int_{E^c }|\p_{B_\nu}(t)|\dd t  &\le  &\int_{E^c } e^{-\sum_{j=1}^\nu 2\theta_j\sin^2\pi t
k_j } \dd t 
\cr &\le  &e^{-\sum_{j=1}^\nu  \theta_j }\int_{E^c }e^{| \sum_{j=1}^\nu 
\theta_j\cos2\pi t k_j | }\dd t 
\cr &\le  &e^{-(1-c)\sum_{j=1}^\nu  \theta_j } .
\end{eqnarray*}
By combining
\begin{eqnarray*}  \int_0^1|\p_{B_\nu}(t)|\dd t  &\le &   \Big(\frac{\p(q)}{c}\Big)^{2q} + e^{-(1-c)\sum_{j=1}^\nu  \theta_j}  . 
\end{eqnarray*}
\end{proof}   
 \begin{remark}Assume $\t_j=\t$ for all $j$. Then
$ \p(q)=\nu^{-1}{\big\|\sum_{j=1}^\nu \cos 2\pi tk_j\big\|_{2q}}$ and further $$\big\|\sum_{j=1}^\nu \cos 2\pi tk_j\big\|_{2q}^{2q}= N_{2q}(\mathcal N),  $$ 
where $N_{2q}(\mathcal N)$ is the number of solutions of \eqref{Deq01} with corresponding set of values $\mathcal N=\{k_1, \ldots, k_\nu \}$.
So that 
\begin{eqnarray*}  \int_0^1|\p_{B_\nu}(t)|\dd t  &\le &   \frac{N_{2q}(\mathcal N)}{(c\nu)^{2q}} + e^{-(1-c)\t(1-\t)\nu   }  . 
\end{eqnarray*}
-- In the case when $\mathcal N=\{1, \ldots, \nu \}$, this together with Theorem \ref{agw} gives  
\begin{eqnarray*}  \int_0^1|\p_{B_\nu}(t)|\dd t  &\le &   \frac{\nu^{2q-1}}{\sqrt q(c\nu)^{2q}} + e^{-(1-c)\t(1-\t)\nu   }\, =\,   \frac{1}{\nu c^{2q}\sqrt q } + e^{-(1-c)\t(1-\t)\nu   }  . 
\end{eqnarray*}
Taking $c=1-(2q)^{-1}$ gives 
\begin{eqnarray*}  \int_0^1|\p_{B_\nu}(t)|\dd t  &\le &  \frac{C}{\nu \sqrt q }+e^{-\t(1-\t)\nu/2q   }    . 
\end{eqnarray*}
 Take $q$ large, $q\sim \t(1-\t)\nu/3\log \nu$. It follows that
\begin{eqnarray*}  \int_0^1|\p_{B_\nu}(t)|\dd t  &\le &  \frac{C\log \nu }{\nu^{3/2} }  \end{eqnarray*}
whereas, in the other hand ${\rm Var}(B_\nu)=\t(1-\t)\sum_{j=1}^\nu j^2\sim C\t(1-\t)\nu^3$    ...

\end{remark} 
\begin{remark}We also have 
\begin{eqnarray*}\sum_{j=1}^\nu \t_j(1- \t_j)\sin^2\pi t k_j
%&\ge& \nu \Big|\prod_{j=1}^\nu \sqrt{  (\t_j(1- \t_j)}\sin \pi t k_j)\Big|^{2/\nu}\cr 
&\ge& \nu \Big(\prod_{j=1}^\nu   \t_j(1- \t_j)\Big)^{1/\nu}\Big|\prod_{j=1}^{k_\nu} \sin \pi t j\Big|^{2/\nu}.
\end{eqnarray*}
 \end{remark}   
%%%%%%%%
\subsection{Proof of Theorem \ref{lltweibin}}
%%%%%%%
Let $ \d=1/2  $. 
%As $\tau^2{ \rm Var}({B_\nu}) = \d L^2$, it follows that 
 By Lemma \ref{lltb},  Lemma \ref{lfp3} and Fourier inversion formula\begin{eqnarray*}
& & \bigg|\P\{{B_\nu}= n\} -\frac{e^{ -   \frac{   ( n-\E {B_\nu})^2}{ 2 {\rm Var}({B_\nu})}  }}{\sqrt{2\pi {\rm Var}({B_\nu})}}   \bigg|
  \le    C\bigg\{  \tau \d^3 +       \frac{  e^{      -  2\pi^2 \tau^2{ \rm Var}({B_\nu}) }     } {   \sqrt{   {\rm Var}  ({B_\nu})   }      }  +  \frac{1}{\sum_{j=1}^\nu 
\t_j(1-\t_j)} \bigg\}
\cr  &\le &  C\bigg\{   \frac{1}{( \sum_{j=1}^\nu (1- \t_j)k^3_j)^{1/3}} +       \frac{  e^{      -  2\pi^2 \d L^2 }     } {  ( \sum_{j=1}^\nu (1- \t_j)k^2_j)^{1/2}      }    +  \frac{1}{\sum_{j=1}^\nu 
\t_j(1-\t_j)} \bigg\}. 
\cr  &\le &         \frac{C}{\sum_{j=1}^\nu 
\t_j(1-\t_j)}. \end{eqnarray*}
Now  similarly, 
\begin{eqnarray*}& & \Big| \P\{{B_\nu}=n_0\}-\frac{1 }{\sqrt{ 2\pi{\rm Var}({B_\nu})}}\Big| \le    \frac{B_2 }{\sqrt{ 2\pi{\rm Var}({B_\nu})}}      +  \frac{1}{\sum_{j=1}^\nu 
\t_j(1-\t_j)}    
\cr &\le  & \frac{C }{( \sum_{j=1}^\nu (1- \t_j)k^2_j)^{1/2}}\, \bigg(\frac{1}{( \sum_{j=1}^\nu (1- \t_j)k^3_j)^{1/3}}+e^{      -  2\pi^2 \d L^2 }    \bigg)
 \cr &   & +  \frac{1}{\sum_{j=1}^\nu 
\t_j(1-\t_j)} 
 \le   \frac{C } {\sum_{j=1}^\nu 
\t_j(1-\t_j)} .\end{eqnarray*}
  This achieves the proof.

%%%%%%%%%%%%%%%
%%%%%%%%%%%%%%%
 \section{An ASLLT related to Burr's problem.}  \label{s4} Let 
 $\l_0<
\l_1<\ldots$ be a sequence of positive integers, call it $A$, and let
$$P(A)=\Big\{ \sum_{i} \e_i \l_i, \hbox{ $\e_i=0$ or $1$, $a_i\in A$ and ${\sum_{i} \e_i<\infty}$}\Big\}. $$
Burr asked in \cite{B} which  sets $S$ of integers are equal to $P(A)$ for some $A$? He
mentioned that if the complement of $S$ grows sufficiently rapidly,
then there exists such a sequence $A$.
  Hegy\' vari   showed in \cite{H} that if $B=\{b_i, i\ge 1\}$ is such that $7 \le b_1<b_2<\cdots $ and
\begin{equation} b_{n+1}\ge 5b_n\qq \hbox{for every $n$},
\end{equation}
then there exists a sequence $A$ such that
 $ P(A)= \N \backslash B$,
thereby improving substantially an earlier unpublished result of
Burr. He also showed that his result cannot be improved essentially. More precisely, if $B$ is such
that
\begin{equation} b_{n+1}\le 2b_n\qq \hbox{for every $n$ large enough},
\end{equation}
and $B$ is a   Sidon set, namely $b_i+b_j= b_k+
b_\ell$ implies $i=k$, $j=\ell$ or $i=\ell$, $j=k$, then there is {\it no}
sequence $A$ for which $ P(A)= \N \backslash B$.  We refer to \cite{{B},{F},{H}} for similar questions.   Here we
examine a variant of the initial problem. Consider the  set
$E$ composed with    all   finite sums   
 \begin{equation}\label{eq1} \l_{j_1} +\ldots+ \l_{j_n}, \qq  0\le j_1 \le \ldots\le j_n, \quad  n\ge 1.
\end{equation}  
 Let $0<\eta<1$ and  let $E_\eta  \subseteq   E$ be the set composed with all finite sums $\l_{j_1} +\ldots+ \l_{j_n}$ such that at most
$\lfloor \eta n \rfloor $ summands may coincide. 

\smallskip Now let
$\{x_n, n\ge 1\}$ be a   sequence of integers increasing nearly linearly, so that it is a relatively "full" sequence.  More precisely,  we
assume  
  there are 
    reals
$a>1,\d>0 $ such that
\begin{equation}\label{21}  x_n -na \sim \d \sqrt{  n} ,\qq\quad n \to \infty.
\end{equation} 
  We are interested in estimating from below the proportion of terms from this sequence which may be represented by a sum $\l_{j_1}
+\ldots+
\l_{j_n}$, namely which belong to   $E$. 
\begin{theorem} Let
$D={\rm g.c.d.}\big\{\l_i-\l_j, i> j\ge 0\big\}
 $. Let also
\begin{equation} \rho= \sup\Big\{r  : \sum_{j=0}^\infty  \l_j\, r^j<\infty\Big\}  .    
\end{equation} 
Assume that $ 0<  \rho\le 1$. Then for some $0<\eta<1$  depending on both $\rho$ and $a$,
 $$ \liminf_{ t\to
\infty}{1\over    \log t } \sum_{ n\le t}  {  1 \over \sqrt n} {\bf 1}_{\{ x_n\in E_\eta\}}\ge {D\over  \sqrt{ 2\pi}\s}e^{- 
{\d^2\over  2\s^2  } } .$$ 
Here we have noted $\s^2=(1-r) \sum_{j=0}^\infty \l_j^2  r^{   j   }-\big((1-r) \sum_{j=0}^\infty \l_j   r^{   j   }\big)^2 $, and $
0<r< 
\rho
$ is   solution of the equation
$$(1-r) \sum_{j=0}^\infty \l_j  r^{   j   }=a.$$
Further,    there exists with probability one a random  subsequence $\l'_j=\l'_j(\o)$, $j=1,\ldots$,  tending to infinity with $n$,
 such that for all $n$ large enough, among $\l'_{ 1 },\ldots, \l'_{ n }$  at most $\lfloor n\eta\rfloor$      may coincide, and 
$$ \lim_{ t\to \infty}{1\over    \log t } \sum_{ n\le
t}  {  1 \over \sqrt n} {\bf 1}_{\{ x_n=\l'_{ 1 } +\ldots+ \l'_{ n } \}} ={D\over  \sqrt{ 2\pi}\s}e^{- 
{\d^2\over  2\s^2  } }.
$$ 
 \end{theorem}
 
%%%%%%%%%%%%%%%%%%%%%%%%
\subsection{Preliminaries}  We   first   recall some  auxiliary results 
on which the proof is based.     Let $X$ be  a square integrable     random variable  
with lattice  distribution function
$F$ and put
  \begin{equation}\label{moment}\m =\E X, \qq\quad \s^2=\E X^2-(\E X)^2 . 
 \end{equation}
Let $D$ be   the maximal
span  of  $X$.   Let also 
$ \{X_k, k\ge 1\}$ be independent copies of
$X$, and consider their partial sums
$S_n=X_1+\ldots +X_n$, $n\ge 1$.  
 We assume throughout that $\s>0$.   
 % Recall        Gnedenko's local limit theorem (\cite{G}  \S 43) 
%\begin{lemma}   \label{Gn}  In order that 
 %$$r(n):= \sup_{N=an+dk}\Big|  \sqrt n \P\{S_n=N\}-{d\over  \sqrt{ 2\pi}\s}%e^{- {(N-n\m )^2/  2 n \s^2} }\Big|=o(1),  $$
% it is necessary and sufficient  that $d=D$.\end{lemma} 
Almost sure versions with rate of Gnedenko's theorem (see after \eqref{llt}) were recently proved in \cite{GW3}.  
  Let   $g(x)= {D\over 
\sqrt{ 2\pi}\s}e^{-  {x^2/(  2\s^2 ) } }  $, $x$ real.  By Gnedenko's local limit theorem,  
\begin{equation}\label{asllt0}\lim_{n\to
\infty}\sqrt n
\P\{S_n=\kappa_n\}
\, =g(\k),  
\end{equation} 
for any sequence $\{\k_n, n\ge 1\}$ of reals such that
\begin{equation}\label{2}  \lim_{n\to \infty} { \k_n-n\m   \over    \sqrt{  n}  }= \k .
\end{equation}  
  We say that $X$ satisfies an almost sure local limit theorem  if 
\begin{equation}\label{asllt}\lim_{N\to \infty}{1\over \log N}\sum_{n=1}^N {1\over \sqrt n}   {\bf 1}_{  \{  S _n   = \kappa_n \}} 
\buildrel{\rm a.s.}\over {=}g(\k),
\end{equation}
 holds whenever    (\ref{2}) is   satisfied.
  It is easily seen that (\ref{asllt}) amounts to establish 
\begin{eqnarray}\label{4}\lim_{N\to \infty}{1\over \log N}\sum_{n=1}^N {B_n\over n}    & \buildrel{\rm
a.s.}\over {=}&0 , 
\end{eqnarray}
where we put $ B_n= \sqrt n \big({\bf 1}_{\{S_n=\kappa_n\}}-\P\{S_n=\kappa_n\} \big)  .$
% In  \cite{GW3} (Theorem 1), we proved the following  almost sure central %limit theorem      with rate.    
\begin{theorem}[\cite{GW3}, Theorem 1] \label{tb}Assume that  $\E X^{2+\e}<\infty$  for some positive $\e$. Then,   
$$ \lim_{ N\to \infty}{1\over    \log N } \sum_{ n\le
N}  {  1 \over \sqrt n} {\bf 1}_{\{S_n=\kappa_n\}} \buildrel{a.s.}\over {=}g(\k),$$
  for any  sequence of integers $\{\k_n, n\ge 1\}$     such that (\ref{2}) holds. 
 Moreover, if (\ref{2}) is sharpened 
as follows,
$$    { \k_n-n\m   \over    \sqrt{  n}  } = \k + \mathcal O_\eta\big(  (\log n)^{-1/2+\eta}        \big),
$$then 
$${1\over    \log N } \sum_{ n\le
N}  {  1 \over \sqrt n} {\bf 1}_{\{S_n=\kappa_n\}}  \buildrel{\rm a.s.}\over {=}g(\k) +\mathcal O_\eta\big(  (\log N)^{-1/2+\eta}      
 \big) \Big).$$ 
\end{theorem}
 
%%%%%%%%%%%
%%%%%%%%%%%
\subsection{Proof} We consider the following random model. Let $0<r<\rho$ and  let $X$ be a  random variable   defined by
$$     \P\{X = \l_j \}  =
   (1-r)  r^{   j   }  ,          \qquad\quad  j= 0,1,\ldots .    
$$ The function  $\m(r)  = (1-r) \sum_{j=0}^\infty \l_j  r^{   j   } $ is continuous on $[0,\rho[$. Further $\m(0)=0$ and $\lim_{r\uparrow
\rho}\m(r)=\infty$. We can thus select a real $r\in ]0,\rho[$ so that  
 $   \E X  = \m(r) =a$. 
 Next  $\E X^2= (1-r) \sum_{j=0}^\infty \l_j^2  r^{   j   } <\infty$. And because $r<\rho$,  $\E X^{2+\a}  <\infty$   for some positive
$\a$.
  It is further clear that $\s^2=\E X^2-  (\E X)^2$ cannot vanish unless $X$ is a constant almost surely, since $\s^2=\E (X -   \E
X)^2$. This case being excluded by construction, we have  $\s>0$.   Let also 
$ \{X_k, k\ge 1\}$ be independent copies of
$X$, and consider their partial sums
$S_n=X_1+\ldots +X_n$, $n\ge 1$
 
\vskip 2pt Now observe that 
\begin{eqnarray*}\P\big\{X_{i_1}=X_{i_2}=\ldots =X_{i_k}\big\}&=& \sum_{j=0}^\infty  \P\{X = \l_j \}^k  =
   (1-r)^k  \sum_{j=0}^\infty r^{ k  j   }\cr &= &{(1-r)^k\over(1-r^k) }  \le 2 (1-r)^k, 
\end{eqnarray*}
if $k$ is large, which we do assume. Thus
\begin{eqnarray*}\P\big\{\exists 1\le i_1<\ldots<i_k\le n:X_{i_1}=X_{i_2}=\ldots =X_{i_k}\big\}&\le & 2 C^k_n (1-r)^k.   
\end{eqnarray*}
We take $k=\lfloor n\eta\rfloor$. Since $n! \sim \sqrt{2\pi n} n^n e^{-n}  $, we have for $n$ large
\begin{eqnarray*}C^k_n
 &\le& 2\Big({n\over n-k}\Big)^{n-k}\Big({n\over  k}\Big)^{ k}\Big({n\over 2\pi (n-k)k}\Big)^{1/2} \le  \Big({1\over 1-\eta}\Big)^{n(1-
\eta ) }\Big({1\over  \eta}\Big)^{ n\eta}\cr &=& \Big[\Big({1\over 1-\eta}\Big)^{  {1\over \eta}- 1 
} \Big({1\over  \eta}\Big)\Big]^{ n\eta} . 
\end{eqnarray*}
Let 
$$ C_n=\Big\{\exists 1\le i_1<\ldots<i_{\lfloor n\eta\rfloor} :X_{i_1}=X_{i_2}=\ldots =X_{i_{\lfloor n\eta\rfloor}}\Big\}, \qq
n=1,2,\ldots$$ 
These sets are non-increasing. And so
 $$\P\Big\{\bigcup_{n=m}^\infty C_n \Big\}\le \P\big\{  C_m \big\}\le   \Big[ 
 \Big({1-r\over  \eta (1-\eta)^{{1\over \eta}-1}}\Big)\Big]^{ m\eta} .$$
  Since $(1-\eta)^{{1\over \eta}-1}= \exp\{-{1-\eta\over \eta}\log {1\over 1-\eta}\} \to 1 $ as   $\eta\uparrow 1$, and $0<r<1$, it
follows that one can select $\eta$ so that
$$  {1-r\over  \eta (1-\eta)^{{1\over \eta}-1}} <1.$$
This choice implies that 
$$  \P\big\{\limsup_{n\to \infty} C_n \big\}=0. $$ 
Thus, with probability one, for all $n$ large enough,   there is no $k$-uple, $1\le i_1<\ldots<i_k$, with $k\ge \lfloor
n\eta\rfloor$,   such that 
$X_{i_1}=X_{i_2}=\ldots =X_{i_k}$. In particular, with probability one, for all $n$ large enough, at most $\lfloor
n\eta\rfloor$ from the random variables $X_i$, $i\le n$ may coincide. 
\vskip 3pt Besides, using  Gnedenko's theorem   we   have, uniformly in
$N$,
 $$     \sqrt n\, \P\{S_n=N\}={D\over   \s\sqrt{ 2\pi }}e^{- 
{(N-na)^2/  2 n \s^2} } +o(1).  $$
By assumption  
 $ \lim_{n\to \infty} { (x_n-na)  /    \sqrt{  n}  }= \d 
 $, so that (\ref{2}) is satisfied. 
Therefore $  \sqrt n\,\P\{S_n=x_n\}\sim {D\over  \s\sqrt{ 2\pi} }e^{- 
{\d^2\over  2\s^2  } }$ as $n\to \infty$. Since the assumptions of   Theorem \ref{tb} are also fulfilled, it follows that
\begin{equation}\label{1} \lim_{ t\to \infty}{1\over    \log t } \sum_{ n\le
t}  {  1 \over \sqrt n} {\bf 1}_{\{S_n=x_n\}} \buildrel{a.s.}\over {=}{D\over  \sqrt{ 2\pi}\s}e^{- 
{\d^2\over  2\s^2  } }.
\end{equation}
By picking  $\o $ in a measurable set of full measure, we find by what proceeds and (\ref{1}), 
that there exists a subsequence $\l_0'\le \l_1'\le \ldots$, $\l_n'=\l_n'(\o)$,
 such that
\begin{equation}\label{10} \lim_{ t\to \infty}{1\over    \log t } \sum_{ n\le
t}  {  1 \over \sqrt n} {\bf 1}_{\{ x_n=\l'_{ 1 } +\ldots+ \l'_{ n } \}} ={D\over  \sqrt{ 2\pi}\s}e^{- 
{\d^2\over  2\s^2  } }.
\end{equation} 
% that there exists an infinite sequence
%of indices
%$j_1(\o),j_2(\o), \ldots$ such that  
%\begin{equation}\label{1} \lim_{ t\to \infty}{1\over    \log t } \sum_{ n\le
%t}  {  1 \over \sqrt n} {\bf 1}_{\{ x_n=\l_{j_1(\o)} +\ldots+ \l_{j_n(\o)}\}} ={D\over  \sqrt{ 2\pi}\s}e^{- 
%{\d^2\over  2\s^2  } }.
%\end{equation} 
%Further, for all $n$ large enough, at most $\lfloor n\eta\rfloor$ from the summands $\l_{j_1(\o)},\ldots, \l_{j_n(\o)}$ may coincide. 
Further, for all $n$ large enough, at most $\lfloor n\eta\rfloor$ from the summands $\l'_{ 1 },\ldots, \l'_{ n }$ may coincide. Thus
  $\l'_n\to
\infty$ with $n$. 
But $  x_n=\l'_{ 1 } +\ldots+ \l'_{ n }$ and 
the fact that among $ \l'_{ 1 },\ldots, \l'_{ n} $, at most   $\lfloor n\eta\rfloor$     can coincide,  implies
that
$x_n\in E_\eta$.
We consequently deduce $$ \liminf_{ t\to
\infty}{1\over   
\log t } \sum_{ n\le t}  {  1 \over \sqrt n} {\bf 1}_{\{ x_n\in E_\eta\}}\ge {D\over  \sqrt{ 2\pi}\s}e^{- 
{x^2\over  2\s^2  } },$$ 
 as claimed.  The second part of the Theorem is a direct consequence of (\ref{10}).  

%%%%%%%%%%%%%%%
%%%%%%%%%%%%%%%
 \section{A Concluding Remark.}\label{s5}
%\subsection{Local limit theorem for classes of infinite sets.} 
A probably well-known fact is that the  local limit theorem is not a sufficiently sharp tool for estimating the probability $\P\{S_n\in E\}$, where    $E$  is an infinite set of integers and  $S_n=X_1+\ldots +X_n$  a sum of independent copies of
a  random variable $X$. As we could not find in the litterature an explicit example,  we mention  here a very simple one given in  \cite{W2} and showing that this already arises  for bounded random variables and for elementary sets $E$, namely arithmetic progressions. 
 
\vskip 2pt
 Let   $d$ be  some positive integer and take $E= d\N$. 
Let also $B_n=\b_1+\ldots+\b_n$,  where  
$
\b_i 
$ are independent standard  Bernoulli random variables. By using the sharpest form   of the local limit theorem for standard Bernoulli random variables, derived from \cite[Theorem 13, Chapter 7]{[P]},  \begin{eqnarray*}  \sup_{z}\, \big|  \P\big\{B_n=z\}
 -\sqrt{\frac{2}{\pi n}} e^{-{ (2z-n)^2\over2 n}}\big|= o ( {1}/{n^{3/2}} ) ,  \end{eqnarray*}
  one easily gets \begin{eqnarray} \label{dllt0}\sup_{2\le d\le n}\Big|\P \{ B_n\in d\N \}-\sqrt{\frac{2}{\pi n}}\sum_{ 
z\equiv 0\, (d)} e^{-{ (2z-n)^2\over 2 n}}\Big|= o(  \frac{\sqrt{  \log n}}{n}  ).
\end{eqnarray}     By operating quite differently, we obtained in \cite{W4} the following uniform estimate.
  Let $\Theta (d,m)
$ be the Theta elliptic function  defined by
\begin{equation*}
\Theta (d,m)  =  \sum_{\ell\in \Z} e^{im\pi{\ell\over   d }-{m\pi^2\ell^2\over 2 d^2}}, 
 \end{equation*}  
   We have  
\begin{equation} \label{uniftheta} \sup_{2\le d\le n}\Big|\P\big\{B_n\in d\N\big\}- {\Theta(d,n)\over d}  \Big|=
{\mathcal O}\Big({ \log^{5/2} n  \over n^{ 3/2}}\Big).
\end{equation} 
By using Poisson summation formula, this implies that $$   \sup_{2\le d\le n}\Big|\P\big\{B_n\in d\N\big\}- \sqrt{  { 2\over \pi
n}}\sum_{ 
z\equiv 0\, (d)} e^{-{ (2z-n)^2\over 2 n}}  \Big|=
{\mathcal O}\Big({ \log^{5/2} n  \over n^{ 3/2}}\Big),
$$
 which is   much better   than   (\ref{dllt0}).  
We refer to \cite{W2} 
 for more details.

%%%%%%%%%%%%%%%
%%%%%%%%%%%%%%%

\end{document}